\documentclass[10pt]{amsart}

\usepackage[T1]{fontenc}
\usepackage{lmodern,amsmath,amssymb,amsthm,mathtools}
\usepackage[margin=1in]{geometry}
\usepackage{tikz-cd,booktabs,array}
\usepackage[shortlabels]{enumitem}
\usepackage{xcolor,needspace}
\usepackage[colorlinks=true,linkcolor=blue,citecolor=blue,urlcolor=blue]{hyperref}

\newtheorem{theorem}{Theorem}[section]
\newtheorem{prop}[theorem]{Proposition}
\newtheorem{lemma}[theorem]{Lemma}
\newtheorem{corollary}[theorem]{Corollary}
\theoremstyle{definition}
\newtheorem{example}[theorem]{Example}
\newtheorem{question}[theorem]{Question}
\theoremstyle{remark}
\newtheorem{remark}[theorem]{Remark}

\newcommand{\nn}{\mathbb{N}}
\newcommand{\zz}{\mathbb{Z}}
\newcommand{\qq}{\mathbb{Q}}
\newcommand{\rr}{\mathbb{R}}
\newcommand{\cc}{\mathbb{C}}
\newcommand{\A}{\mathcal{A}}
\newcommand{\uu}{\mathcal{U}}
\newcommand{\gp}{\mathcal{G}}
\newcommand{\len}{\mathsf{L}}
\newcommand{\fac}{\mathsf{Z}}
\newcommand{\red}{\mathrm{red}}
\newcommand{\mon}{\mathcal{B}}
\newcommand{\exampleqed}{\begingroup\renewcommand{\qedsymbol}{\ensuremath{\blacksquare}}\qed\endgroup}
\DeclareMathOperator{\supp}{supp}
\DeclareMathOperator{\lcm}{lcm}

\title[Additive factorizations in finitely generated semidomains]{Additive factorizations in finitely generated semidomains}

\author{Timothy Chen}
\address{PRIMES-USA, Massachusetts Institute of Technology, Cambridge, MA 02139, USA}
\email{ctimothy@mit.edu}

\author{Felix Gotti}
\address{Department of Mathematics, Massachusetts Institute of Technology, Cambridge, MA 02139, USA}
\email{fgotti@mit.edu}

\author{Alan Yao}
\address{PRIMES-USA, Massachusetts Institute of Technology, Cambridge, MA 02139, USA}
\email{alanyao2008@gmail.com}

\date{October 28, 2026}

\keywords{Additive factorization, semidomain, atomic monoid, half-factorial monoid, finite factorization, ACCP}

\subjclass[2020]{Primary: 20M13; Secondary: 16Y60, 11R04}

\begin{document}

\begin{abstract}
	For finite $S \subseteq \cc$, we study additive factorizations in $\nn_0[S]$ and their behavior under adjoining generators. We construct rational and quadratic examples with infinitely many atoms but without atomicity, including a quadratic example in which every pure power of each generator is an atom. Adjoining generators algebraically independent over $\qq(S)$ preserves atomicity and the standard factorization properties, while unique factorization ascends from the individual monoids under joint linear disjointness for positive algebraic generators. Using the theorem of Evertse, Schlickewei, and Schmidt, we prove that positive real generators give finitely many factorizations of each prescribed length. Thus, bounded and finite factorization are equivalent, and half-factoriality implies finite factorization. We exhibit a set of two quadratic generators that separates half-factoriality from unique factorization and give an augmentation criterion for half-factoriality when every monomial is an atom. Lastly, we characterize complete integral closure for positive rational generators.
\end{abstract}

\maketitle


\bigskip
\section{Introduction}
\label{sec:introduction}

Additive factorizations in semiring extensions depend on the relations among products of their generators. Adjoining a generator can introduce new decompositions and destroy atomicity, even when the individual one-generator monoids are atomic. We investigate these effects for semirings generated by finitely many complex numbers, with particular emphasis on positive real generators.
\smallskip

Let $S = \{s_1, \ldots, s_n\}$ be a finite subset of $\cc$, and set
\[
	\nn_0[S] := \{f(s_1, \ldots, s_n) : f \in \nn_0[x_1, \ldots, x_n]\} \quad \text{and} \quad M_S := (\nn_0[S], +).
\]
Thus, $\nn_0[S]$ is the smallest subsemiring of $\cc$ containing $S$. This semiring is finitely generated, though $M_S$ need not be. We allow additive units and study factorizations up to associates.
\smallskip

The standard factorization properties form the following diagram, adapted from the framework of Anderson, Anderson, and Zafrullah~\cite{AAZ90}:

\begin{figure}[ht]
	\begin{tikzcd}
		\textbf{ UF } \ \arrow[r, Rightarrow] \arrow[red, r, Leftarrow, "/"{anchor=center,sloped}, shift left=1.7ex] \arrow[d, Rightarrow, shift right=1ex] \arrow[red, d, Leftarrow, "/"{anchor=center,sloped}, shift left=1ex] & \ \textbf{ HF } \arrow[d, Rightarrow, shift right=0.6ex] \arrow[red, d, Leftarrow, "/"{anchor=center,sloped}, shift left=1.3ex] \\
		\textbf{ FF } \ \arrow[r, Rightarrow] \arrow[red, r, Leftarrow, "/"{anchor=center,sloped}, shift left=1.7ex] & \ \textbf{ BF } \arrow[r, Rightarrow]  \arrow[red, r, Leftarrow, "/"{anchor=center,sloped}, shift left=1.7ex] & \textbf{ ACCP }  \arrow[r, Rightarrow] \arrow[red, r, Leftarrow, "/"{anchor=center,sloped}, shift left=1.7ex]  & \textbf{ Atomic }
	\end{tikzcd}
	\caption{The implications in the diagram show the known inclusions among six properties of general atomic monoids. The diagram also emphasizes (with red marked arrows) that none of these implications are reversible.}
	\label{fig:diagram}
\end{figure}
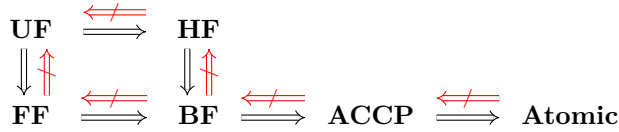
\smallskip

Here UF, HF, FF, and BF denote unique factorization, half-factoriality, finite factorization, and bounded factorization, respectively; ACCP denotes the ascending chain condition on principal ideals. Section~\ref{sec:background} gives their definitions.
\smallskip

Baeth, Chapman, and Gotti~\cite{BCG21} studied additive and multiplicative factorization properties of positive semirings simultaneously. We focus on the additive monoid and ask how its factorization properties depend on generators in the semiring.
\smallskip

Both factorization and divisibility in the one-generator theory have been well studied. For a single positive real generator, Correa-Morris and Gotti~\cite{CG22} proved that the existence of an atom implies atomicity, that ACCP, BF, and FF coincide, and that HF is equivalent to UF. Chen, Gotti, Lu, and Yao~\cite{CGLY26} studied antimatterness and valuation monoids in the algebraic case. Chapman, Gotti, and Gotti~\cite{CGG20} studied factorization invariants for positive rational one-generator semirings. We ask which implications persist for several generators, especially how adjoining generators affects atomicity.
\smallskip

For every finite $S \subseteq \rr_{>0}$, we prove that each element of $M_S$ has only finitely many factorizations of each prescribed length (Theorem~\ref{thm:fixed-length}). The proof applies the theorem of Evertse, Schlickewei, and Schmidt~\cite{ESS02} to sums of monomial atoms. Consequently, BF and FF are equivalent, and HF implies FF. For finite $S \subseteq \qq_{>0}$, elementary weight and chain arguments give more explicit criteria: ACCP, BF, and FF hold exactly when all generators are at least $1$, while HF and UF hold exactly when all generators are positive integers (Theorems~\ref{thm:rational-middle} and~\ref{thm:rational-top}).
\smallskip

Other one-generator conclusions, e.g., that the existence of an atom implies atomicity, no longer apply with two or more generators. For coprime integers $a, b \ge 2$, Theorem~\ref{thm:partial} shows that the atoms of $M_{\{a/b, a/(a + b)\}}$ are precisely the powers of $a/(a + b)$, so $a/b$ has no factorization, and hence $M_{\{a/b, a/(a + b)\}}$ is not atomic. By contrast, when $a > b$, the one-generator monoid $M_{\{a/b\}}$ is an FFM. Example~\ref{ex:axes} gives a quadratic monoid that is not atomic even though every pure power of each generator is an atom. Finally, Example~\ref{ex:quadratic-hf} shows that $M_{\{1 + \sqrt{2}, 1 + 2\sqrt{2}\}}$ is an HFM but not a UFM.
\smallskip

Conversely, we identify settings in which adjoining generators retains factorization properties. For complex generators, Theorem~\ref{thm:transfer} shows that adjoining elements algebraically independent over $\qq(S)$ preserves all six properties in Figure~\ref{fig:diagram}. For positive algebraic generators, Theorem~\ref{thm:uf-ascent} proves UF ascent when the individual monoids are UFMs and their fields are jointly linearly disjoint. Theorem~\ref{thm:augmentation} gives a separate test for half-factoriality when every monomial is an atom.
\smallskip

We also examine complete integral closure for positive rational generators. When $D$ is the least common multiple of their reduced denominators, Theorem~\ref{thm:cic} identifies the complete integral closure with $\zz[1/D]_{\ge 0}$ and shows that $M_S$ is completely integrally closed exactly when $1/D \in M_S$. This specializes the general closure formula for additive submonoids of $\qq_{\ge 0}$ from~\cite{CGG21}.
\smallskip

Section~\ref{sec:background} supplies the definitions. Section~\ref{sec:alg} explores the simplest case when the atoms act independently of one another. Section~\ref{sec:atoms} constructs the failures of atomicity, and Sections~\ref{sec:accp} and~\ref{sec:uf-hf} establish various factorization criteria. Sections~\ref{sec:laurent} and~\ref{sec:cic} treat Laurent evaluations and complete integral closure, respectively. Section~\ref{sec:questions} summarizes the chain of factorization implications in our class of monoids and poses our remaining questions.

\bigskip
\section{Background}
\label{sec:background}

We begin by introducing our terminology, as well as relevant theorems.

\medskip
\subsection{Notation and Semiring Evaluations}

We write $\nn$ for the positive integers, set $\nn_0 := \nn \cup \{0\}$, and use $\zz$, $\qq$, $\rr$, and $\cc$ for the integers, rationals, reals, and complex numbers. For $X \subseteq \rr$, the sets $X_{>0}$ and $X_{\ge 0}$ consist of its positive and nonnegative elements. Numerators and denominators of positive rationals are always coprime and positive; a \emph{unit fraction} has numerator $1$. Empty sums and products equal $0$ and $1$, respectively.
\smallskip

A \emph{semiring} has associative commutative addition, associative multiplication, identities $0$ and $1$, and the distributive laws, with $0$ absorbing under multiplication. A \emph{ring} is a semiring in which every element has an additive inverse. An \emph{integral domain} is a commutative ring with $1 \ne 0$ in which a product of nonzero elements is nonzero. A \emph{field} is a commutative ring with $1 \ne 0$ in which every nonzero element has a multiplicative inverse. Our evaluation semirings are commutative subsemirings of $\cc$, so we refer to them as \emph{semidomains}, i.e., subsemirings of integral domains. Meanwhile, \emph{positive} and \emph{rational} semirings are subsemirings of $\rr_{\ge 0}$ and $\qq$, respectively. A semiring is \emph{finitely generated} if it is the smallest subsemiring containing some finite set. We retain the definitions of $\nn_0[S]$ and $M_S$ from Section~\ref{sec:introduction}.
\smallskip

A \emph{ring homomorphism} preserves addition, multiplication, and $1$; its \emph{kernel} is the inverse image of $0$. For a commutative ring $R$, an \emph{ideal} is an additive subgroup $I$ such that $rI \subseteq I$ for every $r \in R$. The \emph{ideal generated by} $f_1, \ldots, f_d \in R$ is $(f_1, \ldots, f_d) := \{\sum_{i = 1}^d r_i f_i : r_1, \ldots, r_d \in R\}$. The \emph{quotient ring} $R/I$ consists of the cosets modulo $I$, with addition and multiplication induced by $R$.
\smallskip

For $S = \{s_1, \ldots, s_n\}$ and a coefficient ring or semiring $R \subseteq \cc$, the \emph{evaluation map} sends $f \in R[x_1, \ldots, x_n]$ to $f(S) := f(s_1, \ldots, s_n)$; its image is $R[S]$. When $R$ is a ring, its kernel consists of the polynomials vanishing at $S$. The \emph{augmentation map} evaluates a polynomial at $(1, \ldots, 1)$, or equivalently sums its coefficients. We write $\qq(S)$ for the smallest subfield of $\cc$ containing $\qq \cup S$. A \emph{monomial value} in $S$ is an element of
\[
	\mon_S := \{s_1^{e_1} \cdots s_n^{e_n} : (e_1, \ldots, e_n) \in \nn_0^n\}.
\]
We set $s^0 = 1$, including when $s = 0$. The set $\mon_S$ generates $M_S$ additively, though different exponent vectors may give the same value. A \emph{pure power} is a value $s_i^k$ with $k \in \nn_0$, and a \emph{monomial multiple} of $b \in \mon_S$ is a value $bc$ with $c \in \mon_S$. The \emph{support} of a polynomial or indexed family consists of the indices of its nonzero coefficients or coordinates. The \emph{coordinatewise order} on $\nn_0^n$ is defined by $\mathbf{e} \le \mathbf{f}$ if $e_i \le f_i$ for every $i$; a subset is \emph{downward closed} if it contains every vector below each of its elements.
\smallskip

A \emph{Laurent polynomial} over a coefficient ring or semiring $R$ is a finite sum $\sum_{k \in F} c_k X^k$ with $F \subseteq \zz$ finite and $c_k \in R$ for each $k \in F$. For $\alpha \in \rr_{>0}$, the associated \emph{Laurent evaluation monoid} is
\[
	L_\alpha := M_{\{\alpha, \alpha^{-1}\}} = \Big\{\sum_{k \in F} c_k \alpha^k : F \subseteq \zz \text{ is finite and } c_k \in \nn_0 \text{ for every } k \in F\Big\}.
\]
For an integer $D \ge 1$, we write $\zz[1/D] := \{m/D^k : m \in \zz,\ k \in \nn_0\}$. We use $\gcd$ and $\lcm$ for the positive greatest common divisor and least common multiple of nonempty finite sets of positive integers, and set $\lcm(\varnothing) := 1$.

\medskip
\subsection{Atomicity and Factorization}

A \emph{monoid} is a set with an associative operation and an identity, and a \emph{group} is a monoid in which every element has an inverse. Unless specified otherwise, monoids are commutative, cancellative, and written additively with identity $0$. Here \emph{cancellative} means that $x + z = y + z$ implies $x = y$. The \emph{Grothendieck group} $\gp(M)$ consists of formal differences of elements of $M$, and cancellation embeds $M$ in $\gp(M)$. For $M \subseteq (\cc, +)$, we identify $\gp(M)$ with $M - M$. The notation $\langle B \rangle$ denotes the submonoid generated by $B \subseteq M$, consisting of finite sums of elements of $B$, including the empty sum.
\smallskip

The \emph{group of units} of a monoid $M$ is
\[
	\uu(M) := \{u \in M : -u \in M\}.
\]
Elements $x, y \in M$ are \emph{associates} if $x - y \in \uu(M)$. The \emph{reduced monoid} $M_{\red} := M/\uu(M)$ consists of the associate classes, with $[x] + [y] := [x + y]$. The monoid $M$ is \emph{reduced} if $\uu(M) = \{0\}$. An \emph{atom} is a nonunit $a \in M$ such that every decomposition $a = b + c$ with $b, c \in M$ has a unit summand. We denote the set of atoms by $\A(M)$.
\smallskip

The \emph{free commutative monoid} on $B$ consists of formal sums $\sum_{b \in B} n_b b$ with $n_b \in \nn_0$ and finite support, with equality determined by coefficients. Let $\fac(M)$ be the free commutative monoid on $\A(M_{\red})$, and let $\pi_M \colon \fac(M) \to M_{\red}$ send each sum to its value. A \emph{factorization} of $x \in M$ is an element of $\fac_M(x) := \pi_M^{-1}([x])$. For $z = \sum_a n_a a \in \fac(M)$, its \emph{length} is $|z| := \sum_a n_a$, and the \emph{set of lengths} of $x$ is
\[
	\len_M(x) := \{|z| : z \in \fac_M(x)\}.
\]
Factorizations are counted up to permutation and associates. Every unit has only the empty factorization, of length $0$.
\smallskip

A monoid is \emph{atomic} if every nonunit has a factorization, and it is \emph{antimatter} if it has no atoms. An atomic monoid $M$ is a \emph{unique factorization monoid} (UFM), a \emph{half-factorial monoid} (HFM), a \emph{finite factorization monoid} (FFM), or a \emph{bounded factorization monoid} (BFM) if, for every $x \in M$, the set $\fac_M(x)$ is a singleton, the set $\len_M(x)$ is a singleton, the set $\fac_M(x)$ is finite, or the set $\len_M(x)$ is finite, respectively. In particular, atomicity is part of each of these four definitions.
\smallskip

For $x, y \in M$, we say that $y$ \emph{divides} $x$, and write $y \mid_M x$, if $x - y \in M$. The \emph{principal ideal} generated by $x$ is $x + M := \{x + m : m \in M\}$. The \emph{ascending chain condition on principal ideals} (ACCP) requires every ascending chain of principal ideals to stabilize.
\smallskip

The implications UF to HF and FF, and HF and FF to BF, follow from the definitions. If $M$ has the BF property and $x_0 + M \subseteq x_1 + M \subseteq \cdots$, write $x_i = x_{i + 1} + d_i$ with $d_i \in M$. If infinitely many $d_i$ are not units, then we would obtain factorizations of $x_0$ of unbounded length by factoring $x_k$ and the nonunits among $d_0, \ldots, d_{k - 1}$. Hence BF implies ACCP. Under the ACCP, a nonunit $x_0$ without a factorization would split as $x_0 = x_1 + y_1$ with both summands nonunits and, say, $x_1$ without a factorization. Iterating would yield $x_0 + M \subsetneq x_1 + M \subsetneq \cdots$, contradicting the ACCP. Thus, ACCP implies atomicity.
\smallskip

Factorization properties pass between $M$ and $M_{\red}$ by definition. In particular, every group is both atomic and antimatter and, having trivial reduction, satisfies all six properties in Figure~\ref{fig:diagram}. For $S \subseteq \rr_{>0}$, the monoid $M_S$ is reduced and nontrivial, so it cannot be both atomic and antimatter.
\smallskip

A \emph{monoid homomorphism} is a map preserving addition and $0$. For monoids and rings, a bijective homomorphism is an \emph{isomorphism}, and the symbol $\cong$ indicates that such a map exists. An \emph{automorphism} is an isomorphism from an object to itself. An \emph{additive weight} on $M$ is a homomorphism $w \colon M \to \rr_{\ge 0}$, and a \emph{length homomorphism} is a homomorphism $\ell \colon M \to \nn_0$ taking the value $1$ on every atom. For a family $(M_i)_{i \in I}$, its \emph{direct sum} $\bigoplus_{i \in I} M_i$ consists of the tuples with finite support, with coordinatewise addition.
\smallskip

The \emph{complete integral closure} of $M$ in $\gp(M)$ is
\[
	\widehat{M} := \{x \in \gp(M) : \text{there exists } c \in M \text{ such that } c + nx \in M \text{ for every } n \in \nn\}.
\]
The monoid $M$ is \emph{completely integrally closed} if $\widehat{M} = M$.

\medskip
\subsection{Algebraic and Valuation Preliminaries}

An element $\alpha \in \cc$ is \emph{algebraic} over a subfield $K \subseteq \cc$ if a nonzero polynomial in $K[X]$ vanishes at $\alpha$, and \emph{transcendental} otherwise; we omit the subfield $K$ when $K = \qq$. A polynomial is \emph{monic} if its leading coefficient is $1$. A nonconstant polynomial in $K[X]$ is \emph{irreducible} over $K$ if it is not a product of two nonconstant polynomials in $K[X]$. The \emph{minimal polynomial} of an algebraic element $\alpha$ is the unique monic irreducible polynomial in $K[X]$ vanishing at $\alpha$; its roots are the \emph{conjugates} of $\alpha$ over $K$. An algebraic number is \emph{quadratic} if its minimal polynomial over $\qq$ has degree $2$. A \emph{field embedding} is an injective ring homomorphism between fields.
\smallskip

A subset of a vector space over $K$ is \emph{linearly independent} if every finite $K$-linear relation among its elements has all coefficients zero. Its \emph{span} consists of its finite $K$-linear combinations; a linearly independent spanning set is a \emph{basis}, whose cardinality is the \emph{dimension} of the vector space. The \emph{degree} $[E : K]$ of a field extension is the dimension of $E$ over $K$. A finite set $T = \{\tau_1, \ldots, \tau_d\}$ is \emph{algebraically independent} over $K$ if no nonzero polynomial in $K[x_1, \ldots, x_d]$ vanishes at $T$. Algebraic independence therefore implies linear independence of the corresponding monomials. For algebraic numbers $\alpha_1, \ldots, \alpha_d$, the fields $\qq(\alpha_1), \ldots, \qq(\alpha_d)$ are \emph{jointly linearly disjoint} over $\qq$ if
\[
[\qq(\alpha_1, \ldots, \alpha_d) : \qq] = \prod_{i = 1}^d [\qq(\alpha_i) : \qq].
\]
The fields are \emph{pairwise linearly disjoint} if this degree equality holds for each pair; this need not imply joint linear disjointness when $d \ge 3$.
\smallskip

For a prime $p$ and a nonzero rational number $x$, the \emph{$p$-adic valuation} $v_p(x)$ is the unique integer $k$ such that $x = p^k a/b$ with integers $a, b$ not divisible by $p$; we set $v_p(0) := \infty$. It satisfies $v_p(xy) = v_p(x) + v_p(y)$ and the \emph{ultrametric inequality}
\[
	v_p(x + y) \ge \min\{v_p(x), v_p(y)\},
\]
with equality when the two valuations differ. More generally, a \emph{valuation} on a field, as used here, is a map to $\rr \cup \{\infty\}$ that is finite on nonzero elements, sends $0$ to $\infty$, and satisfies these multiplicative and ultrametric rules.
\smallskip

For instance, in Example~\ref{ex:axes}, we use the extension of $v_3$ to $\qq(\sqrt{3})$ given by
\[
	v_3(x + y\sqrt{3}) := \min\{v_3(x), v_3(y) + 1/2\} \qquad (x, y \in \qq).
\]
To verify multiplicativity, observe that for $x + y\sqrt{3} \ne 0$ the displayed value equals $v_3(x^2 - 3y^2)/2$: the finite valuations of $x^2$ and $3y^2$ have different parity, so they cannot cancel. The expression $x^2 - 3y^2$ is multiplicative as a function of $x + y\sqrt{3}$, and the ultrametric inequality follows by applying the rational inequality to the two coordinates. This valuation restricts to the rational $3$-adic valuation and satisfies $v_3(\sqrt{3}) = 1/2$.

\medskip
\subsection{Linear Equations in Multiplicative Groups}

Let $\Gamma$ be a subgroup of $\cc^\times := \cc \setminus \{0\}$ under multiplication, and let $c_1, \ldots, c_k \in \cc^\times$. Following~\cite{ESS02}, a solution $(u_1, \ldots, u_k) \in \Gamma^k$ of the equation
\[
	c_1 u_1 + \cdots + c_k u_k = 1
\]
is \emph{nondegenerate} if $\sum_{i \in I} c_i u_i \ne 0$ for every nonempty subset $I \subseteq \{1, \ldots, k\}$.

\begin{theorem}[Evertse--Schlickewei--Schmidt] \label{thm:ess}
	Let $\Gamma$ be a finitely generated subgroup of $\cc^\times$, and $c_1, \ldots, c_k \in \cc^\times$. Then the equation $c_1 u_1 + \cdots + c_k u_k = 1$ has only finitely many nondegenerate solutions $(u_1, \ldots, u_k) \in \Gamma^k$.
\end{theorem}

This is a special case of~\cite[Theorem~1.1]{ESS02}, which gives an explicit upper bound on the number of such solutions and applies to any algebraically closed field of characteristic $0$.

\bigskip
\section{Algebraic Considerations}
\label{sec:alg}

We identify the Grothendieck group and additive units of $M_S$ and then examine how $M_S$ and its factorization properties behave when new generators are adjoined to~$S$.

\medskip
\subsection{Groups and Units}

We first compute the Grothendieck group and determine when additive units make the entire evaluation monoid a group.

\begin{prop} \label{prop:group}
	Let $S$ be a finite subset of $\cc$. Then $\gp(M_S) = \zz[S]$, and the following conditions are equivalent.
	\begin{enumerate}[(a)]
		\item $1 \in \uu(M_S)$.
		\smallskip
		
		\item $-1 \in M_S$.
		\smallskip
		
		\item $M_S = \zz[S]$.
	\end{enumerate}
	If every element of $S$ is algebraic then these conditions are also equivalent to the condition that $M_S$ is not reduced.
\end{prop}

\begin{proof}
	Splitting integer coefficients into positive and negative parts gives $M_S - M_S = \zz[S]$. Conditions~(a) and~(b) are equivalent by the definition of a unit. If $-1 \in M_S$ then $-x = (-1)x \in M_S$ for every $x \in M_S$, proving~(c). Conversely,~(c) gives $-1 \in M_S$.
	
	Suppose that $S$ consists of algebraic numbers and $0 \ne u \in \uu(M_S)$. Choose an integer polynomial $a_0 + a_1 X + \cdots + a_d X^d$ vanishing at $u$ with $a_0 > 0$. For $1 \le k \le d$, both $u^k$ and $-u^k = (-u)u^{k - 1}$ belong to $M_S$, so $a_k u^k \in M_S$. Hence $-a_0 = \sum_{k = 1}^d a_k u^k \in M_S$, and adding $a_0 - 1$ gives $-1 \in M_S$.
\end{proof}

As the following example demonstrates, the algebraicity hypothesis in Proposition~\ref{prop:group} is necessary.

\begin{example} \label{ex:units}
	Let $t$ be transcendental. Uniqueness of polynomial representations in $t$ yields
	\[
		\nn_0[t, -t] = \nn_0 + t\zz[t] \quad \text{and} \quad \uu(M_{\{t, -t\}}) = t\zz[t].
	\]
	Both signs of each positive power of $t$ occur, while the constant coefficient is nonnegative. Thus, the monoid is neither reduced nor a group, but its reduction is $\nn_0$, so it satisfies all six factorization properties from Figure~\ref{fig:diagram}. \exampleqed
\end{example}

The next criterion detects the group case from the conjugates of a single generator.

\begin{prop} \label{prop:conjugates}
	If $S$ contains a nonzero algebraic number with no positive real conjugate then $M_S$ is a group.
\end{prop}

\begin{proof}
	Let $\alpha \in S$ have no positive real conjugate. Since $\alpha \ne 0$, none of its conjugates is zero. A theorem of Dubickas~\cite{aD07}, after clearing denominators, gives positive integers $c_0, \ldots, c_d$ with $c_0 + c_1 \alpha + \cdots + c_d \alpha^d = 0$. Hence $-c_0 \in M_S$, and adding $c_0 - 1$ gives $-1 \in M_S$, whence we apply Proposition~\ref{prop:group}.
\end{proof}

The converse holds in the one-generator case. The case of $\alpha=0$ or $\alpha$ transcendental are clear, and if $\beta$ were a positive conjugate of $\alpha$, then $M_{\{\alpha\}} \cong M_{\{\beta\}}$ would be a nontrivial positive monoid and therefore not a group. However, it fails with two or more generators, as the following example demonstrates.

\begin{example} \label{ex:units-algebraic}
	Replace the transcendental $t$ in Example~\ref{ex:units} with $\sqrt{2}$. Then we find that $\nn_0[\sqrt{2}, -\sqrt{2}]$ contains $\sqrt{2}(-\sqrt{2}) + 1 = -1$ and is thus a group. However, both $\sqrt{2}$ and $-\sqrt{2}$ have a positive real conjugate. \exampleqed
\end{example}

\medskip
\subsection{Algebraic Independence and Disjointness}

Algebraic independence excludes polynomial relations and therefore gives unique additive factorizations into monomial values.

\begin{prop} \label{prop:independent}
	If $S$ is algebraically independent over $\qq$ then $M_S$ is the free commutative monoid on the monomial basis $\mon_S$. In particular, $\A(M_S) = \mon_S$, and $M_S$ is a UFM.
\end{prop}

\begin{proof}
	The evaluation map $\nn_0[x_1, \ldots, x_n] \to \nn_0[S]$ is surjective. Equal evaluations give an integer polynomial relation, so algebraic independence makes the map injective. Hence $M_S$ is isomorphic to the additive monoid of $\nn_0[x_1, \ldots, x_n]$, which is free on its monomials, and these correspond to $\mon_S$.
\end{proof}

More generally, adjoining independent generators preserves the factorization properties of an arbitrary multi-generator monoid.

\begin{theorem} \label{thm:transfer}
	Let $A$ be a finite subset of $\cc$, and let $T = \{\tau_1, \ldots, \tau_d\}$ with $d \ge 1$ be algebraically independent over $\qq(A)$. Then
	\[
		M_{A \cup T} \cong \bigoplus_{\mathbf{e} \in \nn_0^d} M_A.
	\]
	Moreover, each of atomicity, the ACCP, BF, FF, HF, and UF holds for $M_{A \cup T}$ if and only if it holds for $M_A$.
\end{theorem}

\begin{proof}
	Let $\tau^{\mathbf{e}} := \tau_1^{e_1} \cdots \tau_d^{e_d}$ for $\mathbf{e} \in \nn_0^d$. Every element of $M_{A \cup T}$ is a finite sum $\sum_{\mathbf{e}} a_{\mathbf{e}} \tau^{\mathbf{e}}$ with $a_{\mathbf{e}} \in M_A$. Algebraic independence makes these coefficients unique, proving the isomorphism.
	
	Reduction preserves these properties and commutes with direct sums, so assume that $M_A$ is reduced. An element of the direct sum is an atom if and only if it has exactly one nonzero coordinate and that coordinate is an atom. Factorizations therefore correspond to choices of factorizations in the finitely many nonzero coordinates, and lengths add. This proves the assertions for atomicity, BF, FF, HF, and UF.
	
	For the ACCP, first denote the reduced direct sum by $M$ and then consider the chain $x_0 + M \subseteq x_1 + M \subseteq \cdots$ in $M$. Each $x_i$ divides $x_0$, so its support is contained in $\supp(x_0)$. If $M_A$ satisfies the ACCP then the chains stabilize in these finitely many coordinates, and hence in $M$. Conversely, a nonstabilizing chain in one coordinate gives one in $M$.
\end{proof}

For algebraic generators, the property we now need is joint linear disjointness to ensure that every rational polynomial relation follows from the individual minimal polynomials.

\begin{prop} \label{prop:kernel}
	Let $A = \{\alpha_1, \ldots, \alpha_d\}$ be a set of algebraic numbers, and for each $i \in \{1, \ldots, d\}$ let $m_i$ be the minimal polynomial of $\alpha_i$ and $e_i$ its degree. If
	\[
		[\qq(A) : \qq] = \prod_{i = 1}^d e_i,
	\]
	i.e., the $\qq(\alpha_i)$ are jointly linearly disjoint, then the kernel of the evaluation map $\qq[x_1, \ldots, x_d] \to \qq(A)$ is the ideal $(m_1(x_1), \ldots, m_d(x_d))$.
\end{prop}

\begin{proof}
	The ideal $(m_1(x_1), \ldots, m_d(x_d))$ lies in the kernel. The quotient of $\qq[x_1, \ldots, x_d]$ by this ideal has a basis of monomials $x_1^{r_1} \cdots x_d^{r_d}$ with $0 \le r_i < e_i$ for every $i$, hence dimension $\prod_{i = 1}^d e_i$. The induced surjection onto $\qq(A)$ is an isomorphism by the degree hypothesis.
\end{proof}

Proposition~\ref{prop:kernel} describes what rational polynomial relations exist in such monoids, but does not by itself imply ascent of any factorization properties. However, Theorem~\ref{thm:uf-ascent} uses the same degree condition to prove UF ascent.

\bigskip
\section{Atoms and Failures of Atomicity}
\label{sec:atoms}

We describe the atoms of $M_S$ and construct monoids that have atoms but are not atomic. Throughout this section, generators are positive real numbers unless stated otherwise. We begin by reducing the search for atoms to the search for monomial values.

\begin{lemma} \label{lem:atoms}
	Let $S$ be a finite subset of $\rr_{>0}$. Then the following statements hold.
	\begin{enumerate}[(1)]
		\item Every atom of $M_S$ is a monomial value.
		\smallskip
		
		\item $\A(M_S)$ is nonempty if and only if $1 \in \A(M_S)$.
		\smallskip
		
		\item If a monomial is not an atom of $M_S$ then no monomial multiple of it is an atom of $M_S$.
	\end{enumerate}
\end{lemma}

\begin{proof}
	(1) A polynomial expression with at least two terms, counted with multiplicity, splits into two positive elements. Hence every atom is a monomial value.
	\smallskip
	
	(2) If $1 = u + v$ with $u, v \in M_S \setminus \{0\}$ then every monomial $b$ splits as $b = bu + bv$, so~(1) gives $\A(M_S) = \varnothing$. Conversely, $1 \in \A(M_S)$ implies $\A(M_S) \ne \varnothing$.
	\smallskip
	
	(3) If $b = u + v$ with $u, v \in M_S \setminus \{0\}$ then $bc = uc + vc$ for every monomial $c$.
\end{proof}

Thus, the exponent vectors of monomial atoms form a downward-closed subset of $\nn_0^n$, even when distinct vectors give the same value. However, determining the exact set of atoms remains challenging even for rational generators. We first give a valuation proof of the standard classification for a single rational generator (see, for example,~\cite{CGG20}).

\begin{lemma} \label{lem:rational}
	Let $q = a/b$, where $a, b \in \nn$ are coprime. Then the following statements hold.
	\begin{enumerate}[(1)]
		\item If $b = 1$ then $M_{\{q\}} = \nn_0$.
		\smallskip
		
		\item If $a = 1 < b$ then $M_{\{q\}}$ is antimatter.
		\smallskip
		
		\item If $a, b > 1$ then $\A(M_{\{q\}}) = \{q^k : k \in \nn_0\}$, and so $M_{\{q\}}$ is atomic.
	\end{enumerate}
\end{lemma}

\begin{proof}
	Part~(1) follows from $q \in \nn$, and part~(2) follows from $1 = b \cdot (1/b)$ and Lemma~\ref{lem:atoms}. For part~(3), fix $k \in \nn_0$ and suppose, by way of contradiction, that $q^k$ is a sum of at least two powers of $q$, counted with multiplicity. Each summand $q^j$ is smaller than $q^k$. If $q < 1$, choose a prime $p \mid a$. Then $j > k$, so each summand has valuation $j v_p(a) > k v_p(a) = v_p(q^k)$. If $q > 1$, choose a prime $p \mid b$. Then $j < k$, so each summand has valuation $-j v_p(b) > -k v_p(b) = v_p(q^k)$. In either case, the ultrametric inequality gives $v_p(q^k) > v_p(q^k)$, a contradiction. Thus, every power of $q$ is an atom, and these powers generate $M_{\{q\}}$.
\end{proof}

By contrast, even when the individual one-generator monoids are atomic, their joint semiring can have no additive atoms.

\begin{example} \label{ex:join}
	Both $M_{\{2/3\}}$ and $M_{\{3/4\}}$ are atomic by Lemma~\ref{lem:rational}. The monoid $M_{\{2/3, 3/4\}}$ contains $(2/3)(3/4) = 1/2$, and so every $x \in M_{\{2/3, 3/4\}} \setminus \{0\}$ decomposes as
	\[
		x = \frac{x}{2} + \frac{x}{2},
	\]
	where both summands belong to $M_{\{2/3, 3/4\}}$. Thus, this monoid is antimatter. \exampleqed
\end{example}

Adjoining a generator can also destroy atomicity while leaving infinitely many atoms. The following family separates the existence of atoms from atomicity.

\begin{theorem} \label{thm:partial}
	Let $a, b \ge 2$ be coprime integers, and set $\alpha := a/b$ and $\beta := a/(a + b)$. Then
	\[
		\A(M_{\{\alpha, \beta\}}) = \{\beta^k : k \in \nn_0\},
	\]
	and $M_{\{\alpha, \beta\}}$ is not atomic, although both $M_{\{\alpha\}}$ and $M_{\{\beta\}}$ are atomic.
\end{theorem}

\begin{proof}
	Since $b < a + b$, we have $0 < \beta < \alpha$, and $\beta < 1$. We also have
	\[
		\alpha = \beta(1 + \alpha) = \beta + \alpha\beta.
	\]
	Multiplying by $\alpha^{i - 1}\beta^j$ shows that every monomial $\alpha^i\beta^j$ with $i \ge 1$ is not an atom.
	
	Fix $n \in \nn_0$ and a prime $p \mid a$. Coprimality gives $v_p(\alpha^i \beta^j) = (i + j) v_p(a)$. A monomial $\alpha^i\beta^j < \beta^n$ must satisfy $i + j > n$: otherwise, $\alpha \ge \beta$ and $\beta < 1$ give $\alpha^i\beta^j \ge \beta^{i + j} \ge \beta^n$. Every monomial summand in a nontrivial decomposition of $\beta^n$ would therefore have valuation greater than $n v_p(a) = v_p(\beta^n)$, contradicting the ultrametric inequality. Hence every power of $\beta$ is an atom. No monomial involving $\alpha$ equals a power of $\beta$, since its reduced denominator is divisible by $b$. Lemma~\ref{lem:atoms} now gives $\A(M_{\{\alpha, \beta\}}) = \{\beta^k : k \in \nn_0\}$.
	
	All sums of atoms lie in $\zz[1/(a + b)]$, whereas $\alpha = a/b$ does not, since $b > 1$ and $\gcd(b, a + b) = 1$. Thus, $\alpha$ has no factorization. Both one-generator monoids are atomic by Lemma~\ref{lem:rational}, since their reduced numerators exceed~$1$.
\end{proof}

Two choices of parameters show that coprime denominators do not ensure atomicity and that atomicity can fail after adjoining a generator to an FFM.

\begin{example} \label{ex:two-thirds-two-fifths}
	For $(a, b) = (2, 3)$, Theorem~\ref{thm:partial} gives $\A(M_{\{2/3, 2/5\}}) = \{(2/5)^k : k \in \nn_0\}$, although $M_{\{2/3, 2/5\}}$ is not atomic. Thus, coprime denominators, even with numerator primes dividing neither denominator, do not guarantee atomicity.
	
	For $(a, b) = (3, 2)$, the theorem gives $\A(M_{\{3/2, 3/5\}}) = \{(3/5)^k : k \in \nn_0\}$, again without atomicity. Since $M_{\{3/2\}}$ has the FF property by Theorem~\ref{thm:ge1}, adjoining one rational generator can destroy atomicity while leaving infinitely many atoms. \exampleqed
\end{example}

In Theorem~\ref{thm:partial}, positive powers of one generator cease to be atoms. Over a quadratic field, atomicity can fail even when all pure powers of both generators remain atoms.

\begin{example} \label{ex:axes}
	Set
	\[
		\alpha := \frac{70}{11\sqrt{3}}, \qquad \beta := \frac{\sqrt{3}}{5}, \quad \text{and} \quad r := \alpha\beta = \frac{14}{11}.
	\]
	Consider the monoid $M_{\{\alpha,\beta\}}$. Since $\beta^2 = 3/25$, we have $1 + \beta^2 + r\beta^2 = (275 + 33 + 42)/275 = r$. Multiplying this identity by $\alpha^{i - 1}\beta^{j - 1}$ shows that every monomial $\alpha^i\beta^j$ with $i, j \ge 1$ is not an atom.
	
	The valuation from Section~\ref{sec:background} gives $v_3(\alpha^i\beta^j) = (j - i)/2$ for all $i, j \in \nn_0$. Since $r > 1$ and $\alpha > 1 > \beta$, the inequality $\alpha^i\beta^j = r^i\beta^{j - i} < \beta^m$ implies $\beta^{j - i} < \beta^m$, hence $j - i > m$. Every monomial summand in a nontrivial decomposition of $\beta^m$ would have valuation greater than $m/2 = v_3(\beta^m)$, contradicting the ultrametric inequality. Similarly, $\alpha^i\beta^j = r^j\alpha^{i - j} < \alpha^n$ implies $i - j < n$, so each summand in a nontrivial decomposition of $\alpha^n$ would have valuation greater than $-n/2 = v_3(\alpha^n)$. Hence
	\[
		\A(M_{\{\alpha, \beta\}}) = \{\alpha^n : n \in \nn_0\} \cup \{\beta^m : m \in \nn_0\}.
	\]
	
	However, $r$ has no factorization: since $\alpha > r$, no positive power of $\alpha$ can occur, while sums of powers of $\beta$ lie in $\zz[1/5] + \sqrt{3}\,\zz[1/5]$. The rational elements of this set belong to $\zz[1/5]$, which does not contain $r = 14/11$. \exampleqed
\end{example}

The preceding results distinguish atomicity from the existence of atoms. Below, we also show that it differs from the atomhood of all monomials.

\begin{remark} \label{rem:atomic}
	For positive generators, the following implications hold:
	\[
		\text{every monomial is an atom} \ \Longrightarrow \ M_S \text{ is atomic} \ \Longrightarrow \ 1 \in \A(M_S) \ \Longleftrightarrow \ \A(M_S) \ne \varnothing.
	\]
	Neither of the two reverse implications holds. Atomicity does not force all pure powers to be atoms, as $M_{\{2, 3\}} = \nn_0$ shows. Meanwhile, in Example~\ref{ex:axes}, every pure power of each generator is an atom, but the monoid is not atomic.

    However, if $1$ is an atom of $M_S$, then $1$ must be an atom of every $M_{\{s_i\}}$, and thus each such monoid is atomic. For $s_i \ge 1$, this follows from Theorem~\ref{thm:ge1}. For $0 < s_i < 1$, every nontrivial decomposition of $s_i^k$ uses powers with exponents greater than $k$; dividing by $s_i^k$ would decompose $1$. Hence every power of $s_i$ is an atom. Example~\ref{ex:join} shows that the converse fails: atomicity of each $M_{\{s_i\}}$ does not make $1$ an atom of $M_S$.
\end{remark}

\bigskip
\section{ACCP, Bounded Factorization, and Finite Factorization}
\label{sec:accp}

We prove that BF and FF coincide for finite sets of positive real generators, give elementary weight criteria, and characterize the ACCP, BF, and FF properties for finite sets of positive rational generators.

\begin{theorem} \label{thm:fixed-length}
	Let $S$ be a finite subset of $\rr_{>0}$. Then the following statements hold.
	\begin{enumerate}[(1)]
		\item For every $x \in M_S$ and $k \in \nn_0$, the set $\{z \in \fac_{M_S}(x) : |z| = k\}$ is finite.
		\smallskip
		
		\item The monoid $M_S$ is a BFM if and only if it is an FFM.
		\smallskip
		
		\item If $M_S$ is an HFM then it is an FFM.
	\end{enumerate}
\end{theorem}

\begin{proof}
	For~(1), reducedness handles $x = 0$ or $k = 0$, so assume $x > 0$ and $k \ge 1$. Let $\Gamma := \{s_1^{e_1} \cdots s_n^{e_n} : (e_1, \ldots, e_n) \in \zz^n\}$ be the multiplicative group generated by $S$. By Lemma~\ref{lem:atoms}, every atom lies in $\Gamma$, so each ordered length-$k$ factorization of $x$ gives a tuple $(a_1, \ldots, a_k) \in \Gamma^k$ satisfying
	\[
		\frac{a_1}{x} + \cdots + \frac{a_k}{x} = 1.
	\]
	Positivity makes each such tuple a nondegenerate solution, so Theorem~\ref{thm:ess} with $c_i=1/x$ and $u_i=a_i\in\Gamma$ gives finitely many ordered tuples, hence finitely many length-$k$ factorizations. If $M_S$ satisfies the BF property then its finite length sets and~(1) imply the FF property. The converse in~(2) follows from Figure~\ref{fig:diagram}, while the implication in~(3) follows from combining~(2) with Figure~\ref{fig:diagram}.
\end{proof}

To establish bounded factorization in concrete examples, we use additive weights that bound the number of nonzero summands. The following criterion also applies to reduced monoids outside the evaluation setting and is similar to~\cite[Proposition~4.5]{fG19}.

\begin{prop} \label{prop:weight}
	Let $M$ be a reduced monoid, and let $w \colon M \to \rr_{\ge 0}$ be a monoid homomorphism such that $w(x) \ge \delta$ for some $\delta > 0$ and every $x \in M \setminus \{0\}$. Then $M$ is a BFM, and every factorization of $x \in M$ has length at most $w(x)/\delta$. If, in addition, the set $\{a \in \A(M) : w(a) \le R\}$ is finite for every $R \in \rr_{\ge 0}$, then $M$ is an FFM.
\end{prop}

\begin{proof}
	Every decomposition of $x$ into nonzero summands has at most $w(x)/\delta$ terms. Thus, repeatedly splitting nonatoms terminates in a factorization, proving BF and the stated bound. Under the additional hypothesis, these factorizations use the finite set $\{a \in \A(M) : w(a) \le w(x)\}$ and have bounded length, so there are only finitely many.
\end{proof}

When every generator is at least $1$, the usual real value provides such a weight, and bounded intervals contain only finitely many monomial values.

\begin{theorem} \label{thm:ge1}
	If $S$ is a finite subset of $[1, \infty)$ then $M_S \cap [0, R]$ is finite for every $R \in \rr_{\ge 0}$. In particular, $M_S$ is an FFM.
\end{theorem}

\begin{proof}
	Without loss of generality, assume $1 \notin S$ and $R \ge 1$. For a monomial $s_1^{e_1} \cdots s_n^{e_n} \le R$, each exponent satisfies $e_i \le \log R/\log s_i$, leaving finitely many monomial values at most $R$. Each is at least $1$, so a polynomial expression with value at most $R$ has at most $\lfloor R \rfloor$ terms. Hence $M_S \cap [0, R]$ is finite, and Proposition~\ref{prop:weight} with $w(x) = x$ and $\delta = 1$ gives FF.
\end{proof}

Gotti proved that every additive submonoid of $\rr_{\ge 0}$ generated by an increasing sequence is an FFM~\cite[Theorem~5.6]{fG19}. This also implies the FF conclusion of Theorem~\ref{thm:ge1}, since finiteness in bounded intervals lets us list the positive elements of $M_S$ in increasing order. Our proof establishes the stronger assertion of finiteness in every bounded interval. The same conclusion holds when a field embedding sends all generators into $[1, \infty)$.

\begin{corollary} \label{cor:embedding}
	If there is a field embedding $\sigma \colon \qq(S) \to \rr$ such that $\sigma(s) \ge 1$ for every $s \in S$ then $M_S$ is an FFM.
\end{corollary}

\begin{proof}
	The restriction of $\sigma$ gives $M_S \cong M_{\sigma(S)}$, so Theorem~\ref{thm:ge1} applies.
\end{proof}

We now pivot to the ACCP, whose failure of course requires sequences of elements that converge to $0$. We demonstrate that certain decompositions yield these sequences and, hence, an infinite ascending chain of principal ideals.

\begin{prop} \label{prop:chain}
	Let $S$ be a finite subset of $\rr_{>0}$. If there are $t, c, d \in M_S \setminus \{0\}$ such that $c = ct + d$ then $M_S$ does not satisfy the ACCP.
\end{prop}

\begin{proof}
	Positivity and $c = ct + d$ give $0 < t < 1$. Set $x_k := ct^k$ for $k \in \nn_0$. Then $x_k = x_{k + 1} + dt^k$ with $dt^k \in M_S \setminus \{0\}$, yielding
	\[
		x_0 + M_S \subsetneq x_1 + M_S \subsetneq x_2 + M_S \subsetneq \cdots.
	\]
	The inclusions are strict because $x_{k + 1} - x_k < 0$ and hence $x_{k + 1} - x_k \notin M_S$.
\end{proof}

Every rational element of $M_S$ strictly between $0$ and $1$ produces a relation of this form.

\begin{corollary} \label{cor:rational-below-one}
	Let $S$ be a finite subset of $\rr_{>0}$. If $M_S$ contains a rational number in $(0, 1)$ then $M_S$ does not satisfy the ACCP, and so $M_S$ is neither a BFM nor an FFM.
\end{corollary}

\begin{proof}
	Write $t = a/b \in M_S$ with $a, b \in \nn$ and $a < b$. Then $c := a$ and $d := (b - a)t$ are nonzero elements of $M_S$ satisfying $ct + d = bt = a = c$, so Proposition~\ref{prop:chain} applies.
\end{proof}

The weight and chain criteria give a complete characterization of the ACCP, BF, and FF properties for positive rational generators.

\begin{theorem} \label{thm:rational-middle}
	Let $S$ be a finite subset of $\qq_{>0}$. Then the following conditions are equivalent.
	\begin{enumerate}[(a)]
		\item $M_S$ is an FFM.
		\smallskip
		
		\item $M_S$ is a BFM.
		\smallskip
		
		\item $M_S$ satisfies the ACCP.
		\smallskip
		
		\item $S \subseteq [1, \infty)$.
	\end{enumerate}
\end{theorem}

\begin{proof}
	The implications (a)~$\Rightarrow$~(b)~$\Rightarrow$~(c) follow from Figure~\ref{fig:diagram}, and (d)~$\Rightarrow$~(a) follows from Theorem~\ref{thm:ge1}. If~(d) fails then $S$ contains a rational number in $(0, 1)$, so Corollary~\ref{cor:rational-below-one} rules out~(c).
\end{proof}

The ACCP is strictly stronger than atomicity, even for one rational generator. The following case also exhibits unbounded factorization lengths; it is a standard example from both~\cite[Corollary~4.4]{CGG21} and~\cite[Corollary~4.8]{CG22}.

\begin{example} \label{ex:atomic-notaccp}
	For $q := 2/3$, Lemma~\ref{lem:rational} gives atomicity, while $x_k := 2q^k$ gives the chain from Corollary~\ref{cor:rational-below-one}. Starting with $2 = 3q$ and iterating $3q^k = q^k + 2q^k = q^k + 3q^{k + 1}$ yields, for each $k \in \nn$, the factorization
	\[
		2 = q + q^2 + \cdots + q^{k - 1} + 3q^k,
	\]
	where the initial sum is empty for $k = 1$. The length is $k + 2$, so the lengths of factorizations of $2$ are unbounded.
	\exampleqed
\end{example}

\bigskip
\section{Factoriality and Half-Factoriality}
\label{sec:uf-hf}

We characterize the UF property through linear independence of atoms and half-factoriality through length homomorphisms.

\medskip
\subsection{Atomic Bases}

For reduced atomic submonoids of $\cc$, unique factorization is equivalent to linear independence of the atoms.

\begin{prop} \label{prop:uf-basis}
	A reduced atomic submonoid $M$ of $(\cc, +)$ is a UFM if and only if its atoms are linearly independent over $\qq$.
\end{prop}

\begin{proof}
	Linear independence of the atoms over $\qq$ forces two factorizations of the same element to coincide. Conversely, a nontrivial rational relation among atoms gives, after clearing denominators and separating terms by sign, an equality of two sums of atoms with disjoint supports. These represent distinct factorizations, contradicting UF.
\end{proof}

Algebraic independence of the semiring generators, which leads to linear independence of the atoms, is sufficient for UF but not necessary: if $t$ is transcendental then $M_{\{t, t^2\}} = M_{\{t\}}$ is a UFM. Similarly, $M_{\{\sqrt{2}\}} = \nn_0 + \nn_0\sqrt{2}$ is a UFM with atoms $1$ and $\sqrt{2}$.
\smallskip

For positive algebraic generators, joint linear disjointness is what we need for linear independence, as with Proposition~\ref{prop:kernel}.

\begin{theorem} \label{thm:uf-ascent}
	Let $\alpha_1, \ldots, \alpha_d$ be positive algebraic numbers such that the fields $\qq(\alpha_1), \ldots, \qq(\alpha_d)$ are jointly linearly disjoint. If $M_{\{\alpha_i\}}$ is a UFM for every $i \in \{1, \ldots, d\}$ then $M_{\{\alpha_1, \ldots, \alpha_d\}}$ is a UFM.
\end{theorem}

\begin{proof}
	For each $i \in \{1, \ldots, d\}$, set $B_i := \A(M_{\{\alpha_i\}})$. Proposition~\ref{prop:uf-basis} makes $B_i$ linearly independent over $\qq$. Atomicity shows that $B_i$ spans every power of $\alpha_i$, hence $\qq(\alpha_i)$. Thus, $B_i$ is a basis for $\qq(\alpha_i)$ and has cardinality $[\qq(\alpha_i) : \qq]$.
	
	Set $B := \{b_1 \cdots b_d : b_i \in B_i \text{ for every } i\}$. Expanding the atomic decompositions of the powers $\alpha_i^{e_i}$ shows that $B$ generates $M_{\{\alpha_1, \ldots, \alpha_d\}}$ additively and spans $\qq(\alpha_1, \ldots, \alpha_d)$. Joint linear disjointness gives $[\qq(\alpha_1, \ldots, \alpha_d) : \qq] = \prod_{i = 1}^d |B_i| \ge |B|$. Therefore $M_{\{\alpha_1, \ldots, \alpha_d\}}$ is free on $B$ and is a UFM.
\end{proof}

Note that for a single positive algebraic generator, HF and UF are equivalent and hold precisely when its minimal polynomial has the form $X^d - \sum_{j = 0}^{d-1} c_j X^j$ with $c_j \in \nn_0$ for every $j$~\cite[Theorem~5.4(2)]{CG22}.

\medskip
\subsection{Length Homomorphisms}

For half-factoriality, relations among atoms are allowed, but they must preserve the total number of summands. We can express this requirement with length homomorphisms.

\begin{prop} \label{prop:hf-length}
	A reduced atomic monoid $M$ is an HFM if and only if there is a monoid homomorphism $\ell \colon M \to \nn_0$ such that $\ell(a) = 1$ for every $a \in \A(M)$.
\end{prop}

\begin{proof}
	If $M$ is an HFM then the map sending each element to the common length of its factorizations is such a homomorphism. Conversely, if $\ell$ exists then additivity forces every factorization of $x \in M$ to have length $\ell(x)$.
\end{proof}

When every monomial value is an atom, augmentation is the natural candidate for a length homomorphism. Its compatibility with polynomial relations gives the following criterion.

\Needspace{11\baselineskip}
\begin{theorem} \label{thm:augmentation}
	Let $S = \{s_1, \ldots, s_n\}$ be a finite subset of $\cc$, and let $I_S$ be the kernel of the evaluation map $\zz[x_1, \ldots, x_n] \to \zz[S]$. Then the following conditions are equivalent.
	\begin{enumerate}[(a)]
		\item Every $h \in I_S$ satisfies $h(1, \ldots, 1) = 0$.
		\smallskip
		
		\item The assignment $\ell(f(S)) := f(1, \ldots, 1)$ for $f \in \nn_0[x_1, \ldots, x_n]$ is a well-defined map $\ell \colon M_S \to \nn_0$.
		\smallskip
		
		\item $M_S$ is a reduced HFM, and every monomial value $s_1^{e_1} \cdots s_n^{e_n}$ is an atom of $M_S$.
	\end{enumerate}
\end{theorem}

\begin{proof}
	The equivalence of~(a) and~(b) follows from writing each polynomial with integer coefficients as the difference of two polynomials with nonnegative integer coefficients.
	
	Assume~(b). If $f(S) = 0$ with $f \in \nn_0[x_1, \ldots, x_n]$ then $f(1, \ldots, 1) = \ell(0) = 0$, so $f = 0$. Thus, the additive map $\ell$ is positive on every nonzero element, and $M_S$ is reduced. Every monomial value has $\ell$-value $1$, so it is an atom, and the converse follows from reducedness. Since they generate $M_S$, it is atomic and an HFM by Proposition~\ref{prop:hf-length}.
	
	Assume~(c), and write $h \in I_S$ as $h = f - g$ with $f, g \in \nn_0[x_1, \ldots, x_n]$. The expressions $f(S)$ and $g(S)$ factor the same element into monomial atoms. Half-factoriality equates their lengths $f(1, \ldots, 1)$ and $g(1, \ldots, 1)$, so $h(1, \ldots, 1) = 0$.
\end{proof}

The criterion requires every monomial value to be an atom for well-definedness. For example, $M_{\{2\}} = \nn_0$ is an HFM, but its monomial value $2$ is not an atom. We now use augmentation to construct an HFM that does not have unique factorization.

\begin{example} \label{ex:hf}
	Let $t > 1$ be transcendental, and set $u := t^2 - t + 1$. If $h(t, u) = 0$ for $h \in \zz[x_1, x_2]$ then transcendence gives $h(X, X^2 - X + 1) = 0$ in $\zz[X]$. Evaluating at $X = 1$ yields $h(1, 1) = 0$. By Theorem~\ref{thm:augmentation}, the monoid $M_{\{t, u\}}$ is an HFM and every monomial value is an atom.
	
	Since $1 < t < u < t^2$, the equality
	\[
		t + u = t^2 + 1
	\]
	gives two distinct factorizations of length $2$, so $M_{\{t, u\}}$ is not a UFM. Both generators exceed $1$, so it is an FFM by Theorem~\ref{thm:ge1}. \exampleqed
\end{example}

The next example shows that HF need not imply UF even for positive algebraic generators. We use a different length homomorphism.

\begin{example} \label{ex:quadratic-hf}
	Set $a := 1 + \sqrt{2}$ and $b := 1 + 2\sqrt{2}$. The additive monoid $M_0 := \langle 1, a, b \rangle$ is closed under multiplication because
	\[
		a^2 = 2 + b, \qquad ab = 3 + a + b, \quad \text{and} \quad b^2 = 7 + 2b.
	\]
	Hence $M_{\{a, b\}} = M_0$. For $r, s \in \qq$, the rational-coordinate map $r + s\sqrt{2} \mapsto r$ restricts to a homomorphism $\ell \colon M_0 \to \nn_0$. Each expression as a sum of $1$, $a$, and $b$ has $\ell$-value equal to its number of summands. Thus, $\ell$ is positive on nonzero elements, and these three distinct generators, each of $\ell$-value $1$, are precisely the atoms. Proposition~\ref{prop:hf-length} gives HF, whereas
	\[
		1 + b = a + a
	\]
	gives two distinct factorizations, so $M_0$ is not a UFM. Since $a, b > 1$, it is an FFM by Theorem~\ref{thm:ge1}. \exampleqed
\end{example}

The next example separates finite factorization from half-factoriality with only one rational generator.

\begin{example} \label{ex:ff-not-hf}
	The monoid $M_{\{3/2\}}$ is an FFM by Theorem~\ref{thm:ge1}, and $1$ and $3/2$ are atoms by Lemma~\ref{lem:rational}. The equality
	\[
		1 + 1 + 1 = \frac{3}{2} + \frac{3}{2}
	\]
	shows that $M_{\{3/2\}}$ is not an HFM. 
	\exampleqed
\end{example}

For positive rational generators, half-factoriality forces the additive monoid to be $\nn_0$. This gives the following characterization of unique- and half-factoriality, mirroring~\cite[Proposition~4.22]{CGG21} for the one-generator case.

\begin{theorem} \label{thm:rational-top}
	Let $S$ be a finite subset of $\qq_{>0}$. Then the following conditions are equivalent.
	\begin{enumerate}[(a)]
		\item $M_S$ is a UFM.
		\smallskip
		
		\item $M_S$ is an HFM.
		\smallskip
		
		\item $M_S = \nn_0$.
		\smallskip
		
		\item $S \subseteq \nn$.
	\end{enumerate}
\end{theorem}

\begin{proof}
	The implication (a)~$\Rightarrow$~(b) follows from Figure~\ref{fig:diagram}. For (b)~$\Rightarrow$~(c), suppose that $M_S$ is an HFM. Since $M_S$ is atomic and not a group, Lemma~\ref{lem:atoms} makes $1$ an atom. For an atom $c/d$ with coprime $c, d \in \nn$, the equality $d \cdot (c/d) = c \cdot 1$ gives lengths $d$ and $c$. Hence $c = d$, so $1$ is the only atom and $M_S = \nn_0$. The inclusion $S \subseteq M_S$ gives (c)~$\Rightarrow$~(d). Finally, (d) gives $M_S = \nn_0$, which is a UFM, proving (a).
\end{proof}

Combining Theorems~\ref{thm:rational-middle} and~\ref{thm:rational-top}, we obtain the following implications for every finite subset $S$ of $\qq_{>0}$:

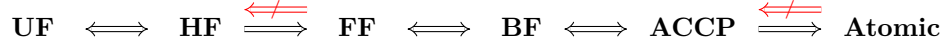
\begin{figure}[ht]
	\begin{tikzcd}
		\textbf{ UF } \ \arrow[r, Leftrightarrow] & \ \textbf{ HF } \arrow[r, Rightarrow] \arrow[red, r, Leftarrow, "/"{anchor=center,sloped}, shift left=1.7ex] & \ \textbf{ FF } \ \arrow[r, Leftrightarrow]	& \ \textbf{ BF } \arrow[r, Leftrightarrow] & \textbf{ ACCP } \arrow[r, Rightarrow] \arrow[red, r, Leftarrow, "/"{anchor=center,sloped}, shift left=1.7ex]  & \textbf{ Atomic }
	\end{tikzcd}
	\caption{The diagram shows implications between the six properties in Figure~\ref{fig:diagram} restricted to the class of additive monoids of finitely generated positive rational semirings. Red arrows mark implications that are not reversible, shown by Examples~\ref{ex:ff-not-hf} and~\ref{ex:atomic-notaccp}.}
\end{figure}

\bigskip
\section{Laurent Evaluations}
\label{sec:laurent}

We compare our results with the Laurent evaluation monoids $L_\alpha = M_{\{\alpha, \alpha^{-1}\}}$. These were studied by Zhu~\cite{sZ22}, and we include the relevant statements and proofs for completeness.

\begin{prop}[{\cite[Theorem~3.1 and Proposition~3.4]{sZ22}}] \label{prop:laurent}
	Let $\alpha \in \rr_{>0}$. Then the following statements hold.
	\begin{enumerate}[(1)]
		\item Either $L_\alpha$ is antimatter, or $\A(L_\alpha) = \{\alpha^k : k \in \zz\}$ and $L_\alpha$ is atomic.
		\smallskip
		
		\item If $\alpha \ne 1$ then $L_\alpha$ is atomic if and only if $\A(M_{\{\alpha\}})$ is infinite.
	\end{enumerate}
\end{prop}

\begin{proof}
	(1) Every atom of $L_\alpha$ is a Laurent monomial value, as in Lemma~\ref{lem:atoms}. Multiplication by $\alpha$ is an automorphism, so the existence of one atom makes every integer power of $\alpha$ an atom. These powers generate $L_\alpha$.
	\smallskip
	
	(2) If $L_\alpha$ is atomic then every nonnegative power of $\alpha$ remains an atom in $M_{\{\alpha\}}$ and these powers are distinct because $\alpha \ne 1$. Conversely, if $\A(M_{\{\alpha\}})$ is infinite then every nonnegative power of $\alpha$ is an atom: by Lemma~\ref{lem:atoms}(3), if one power were not an atom then no higher power would be an atom. A decomposition $1 = u + v$ with $u, v \in L_\alpha \setminus \{0\}$ would, after multiplication by $\alpha^N$ for sufficiently large $N \in \nn$, decompose $\alpha^N$ in $M_{\{\alpha\}}$. This is impossible, so $1$ is an atom of $L_\alpha$, and~(1) gives atomicity.
\end{proof}

The description of Laurent atoms also determines when half-factoriality and unique factorization hold.

\begin{prop}[{\cite[Proposition~5.1]{sZ22}}] \label{prop:laurent-uf}
	Let $\alpha \in \rr_{>0}$. Then the following conditions are equivalent.
	\begin{enumerate}[(a)]
		\item $L_\alpha$ is a UFM.
		\smallskip
		
		\item $L_\alpha$ is an HFM.
		\smallskip
		
		\item Either $\alpha = 1$ or $\alpha$ is transcendental.
	\end{enumerate}
\end{prop}

\begin{proof}
	The implication (a)~$\Rightarrow$~(b) is part of Figure~\ref{fig:diagram}. For (c)~$\Rightarrow$~(a), note that $L_1 = \nn_0$ and that, for transcendental $\alpha$, uniqueness of Laurent polynomial representations makes $L_\alpha$ free on the integer powers of $\alpha$.
	
	For (b)~$\Rightarrow$~(c), suppose, by way of contradiction, that $L_\alpha$ is an HFM with $\alpha \ne 1$ algebraic. Proposition~\ref{prop:laurent} makes every integer power of $\alpha$ an atom. Write a nonzero integer multiple of its minimal polynomial as $p - q$ with $p, q \in \nn_0[X]$ having disjoint supports. Positivity of $\alpha$ makes both polynomials nonzero. The equality $p(\alpha) = q(\alpha)$ gives factorizations of lengths $p(1)$ and $q(1)$. Half-factoriality forces $p(1) = q(1)$, so $1$ is a root of the irreducible minimal polynomial, contrary to $\alpha \ne 1$.
\end{proof}

Zhu proved that FF, BF, and the ACCP are equivalent for positive Laurent evaluations~\cite[Theorem~4.3]{sZ22}. Beyond Theorem~\ref{thm:fixed-length}, this supplies the implication from the ACCP to BF. Below we exhibit a Laurent FF monoid, a specific case of~\cite[Theorem~4.7]{sZ22}, whose prescribed generators include an element below $1$.

\begin{example} \label{ex:laurent-ff}
	Set $\alpha := (3 + \sqrt{5})/2$, whose conjugate is $\alpha^{-1} = (3 - \sqrt{5})/2$. For the nontrivial automorphism $\sigma$ of $\qq(\sqrt{5})$, define $w \colon L_\alpha \to \rr_{\ge 0}$ by $w(x) := x + \sigma(x)$. For every $k \in \zz$, we have $w(\alpha^k) = \alpha^k + \alpha^{-k} \ge 2$. Every nonzero element therefore has weight at least $w(1) = 2$, so $1$ is an atom. By Proposition~\ref{prop:laurent}, the atoms are precisely the powers $\alpha^k$ with $k \in \zz$. Their weights tend to infinity as $|k| \to \infty$, so Proposition~\ref{prop:weight} gives FF. However, $\alpha + \alpha^{-1} = 1 + 1 + 1$ has factorizations of distinct lengths $2$ and $3$, so $L_\alpha$ is not an HFM. Since $\alpha^{-1} < 1$, the condition $S \subseteq [1, \infty)$ in the rational characterization of Theorem~\ref{thm:rational-middle} is not necessary for positive algebraic generators. 
	\exampleqed
\end{example}

Finally,~\cite[Example~3.5]{sZ22} shows that, when $a,b\in\nn\setminus\{1\}$ and $\gcd(a,b)=1$, the Laurent evaluation monoid $L_{a/b}$ is atomic but does not satisfy the ACCP. Combining Zhu's results, we see that in the Laurent case, the six properties satisfy the following series of implications:

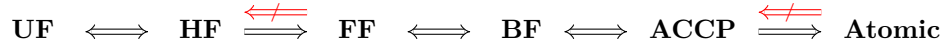
\begin{figure}[ht]
	\begin{tikzcd}
		\textbf{ UF } \ \arrow[r, Leftrightarrow] & \ \textbf{ HF } \arrow[r, Rightarrow] \arrow[red, r, Leftarrow, "/"{anchor=center,sloped}, shift left=1.7ex] & \ \textbf{ FF } \ \arrow[r, Leftrightarrow]	& \ \textbf{ BF } \arrow[r, Leftrightarrow] & \textbf{ ACCP } \arrow[r, Rightarrow] \arrow[red, r, Leftarrow, "/"{anchor=center,sloped}, shift left=1.7ex]  & \textbf{ Atomic }
	\end{tikzcd}
	\caption{The diagram shows implications among the six properties in Figure~\ref{fig:diagram} for the class of Laurent evaluation monoids. Red arrows mark implications that are not reversible.}
\end{figure}

\bigskip
\section{Complete Integral Closure in the Rational Setting}
\label{sec:cic}

For finite $S \subseteq \qq_{>0}$, we determine $\widehat{M_S}$ and characterize complete integral closure using the following known formula~\cite[Proposition~3.5(1)]{CGG21}. We include an elementary proof.

\begin{prop} \label{prop:cic-cone}
	Let $M$ be an additive submonoid of $\qq_{\ge 0}$. Then $\widehat{M} = \gp(M) \cap \qq_{\ge 0}$.
\end{prop}

\begin{proof}
	No negative $x \in \gp(M)$ belongs to $\widehat{M}$ because $c + nx$ is eventually negative for each $c \in M$. Also, $0 \in \widehat{M}$. For positive $x \in \gp(M)$, write $x = a - b$ with $a, b \in M$, and set $M_0 := \langle a, b \rangle \subseteq M$. Then $\gp(M_0) = \delta\zz$ for some $\delta \in \qq_{>0}$. Scaling by $1/\delta$ gives a monoid generated by coprime nonnegative integers, so $M_0$ contains every sufficiently large multiple of $\delta$. Choose $c \in M_0$ above this threshold. Since $x \in \delta\zz_{>0}$, it follows that $c + nx \in M_0 \subseteq M$ for every $n \in \nn$.
\end{proof}

For a finitely generated rational semiring, the denominators determine its Grothendieck group. The formula above then reduces complete integral closure to membership of a single unit fraction.

\begin{theorem} \label{thm:cic}
	Let $S = \{a_1/b_1, \ldots, a_n/b_n\}$ be a finite subset of $\qq_{>0}$, where $a_i, b_i \in \nn$ are coprime for every $i \in \{1, \ldots, n\}$, and set $D := \lcm(b_1, \ldots, b_n)$. Then $\gp(M_S) = \zz[1/D]$, and the following conditions are equivalent.
	\begin{enumerate}[(a)]
		\item $M_S$ is completely integrally closed.
		\smallskip
		
		\item $M_S = \zz[1/D]_{\ge 0}$.
		\smallskip
		
		\item $1/D \in M_S$.
	\end{enumerate}
	Moreover, if $D > 1$ and these conditions hold then $M_S$ is antimatter.
\end{theorem}

\begin{proof}
	For each $i \in \{1, \ldots, n\}$, B\'ezout's identity gives $u, v \in \zz$ with $u a_i + v b_i = 1$, so $1/b_i = u(a_i/b_i) + v \in \zz[S]$. Multiplying $\prod_{i = 1}^n (1/b_i)$ by the integer $(\prod_{i = 1}^n b_i)/D$ gives $1/D \in \zz[S]$. Since $b_i \mid D$ for every $i$, we obtain $\zz[S] = \zz[1/D]$, and Proposition~\ref{prop:group} gives $\gp(M_S) = \zz[1/D]$.
	
	Proposition~\ref{prop:cic-cone} gives the equivalence of~(a) and~(b), and~(b) implies~(c). If~(c) holds then $m/D^k = m(1/D)^k \in M_S$ for all $m, k \in \nn_0$, giving~(b), since the reverse inclusion always holds. Finally, if $D > 1$ and~(c) holds then every positive $x \in M_S$ is the sum of $D$ copies of $x/D = x \cdot (1/D) \in M_S$, so $M_S$ is antimatter.
\end{proof}

Note that the theorem can be extended to apply to negative generators. When $M_S$ is a group, it provides a concise characterization of the monoid. Specifically, if a finite set $S$ of rationals contains a negative number, then $M_S = \zz[S]$ is a group by Proposition~\ref{prop:conjugates}. Hence $M_S = \zz[1/D]$, where $D$ is the least common multiple of the reduced denominators.
\smallskip

When a prime divides every reduced numerator, the required unit fraction is not in $M_S$, which then fails to be completely integrally closed.

\begin{corollary} \label{cor:common-numerator}
	Let $S$ be a finite subset of $\qq_{>0}$ with $S \not\subseteq \nn$. If some prime $p$ divides the reduced numerator of every element of $S$ then $M_S$ is not completely integrally closed.
\end{corollary}

\begin{proof}
	Let $D$ be the least common multiple of the reduced denominators of the elements of $S$. Since $S \not\subseteq \nn$, we have $D > 1$; coprimality of each reduced numerator and denominator gives $p \nmid D$. If $1/D = f(S)$ for a polynomial $f$ with nonnegative integer coefficients then $1/D < 1$ forces its constant term to be $0$. Every nonconstant monomial value has positive $p$-adic valuation, so the ultrametric inequality gives $v_p(1/D) > 0$, contradicting $p \nmid D$. Thus, $1/D \notin M_S$, and Theorem~\ref{thm:cic} applies.
\end{proof}

For positive generators, adjoining a suitable unit fraction can supply the denominators needed for complete integral closure.

\begin{corollary} \label{cor:unit-fraction}
	Let $q = a/b$, where $a, b \in \nn$ are coprime, and let $n \ge 2$ be an integer. If every prime divisor of $a$ divides $n$, then $M_{\{q, 1/n\}}$ is completely integrally closed.
\end{corollary}

\begin{proof}
	Choose $k \in \nn_0$ such that $a \mid n^k$. Then $(n^k/a)q(1/n)^k = 1/b$ belongs to $M_{\{q, 1/n\}}$, so $1/(bn)$ also belongs to this monoid. Multiplying by the integer $bn/\lcm(b, n)$ gives $1/\lcm(b, n) \in M_{\{q, 1/n\}}$, and Theorem~\ref{thm:cic} applies.
\end{proof}

The following example illustrates how a sum of the original generators can produce a unit fraction that generates the entire semiring.

\begin{example} \label{ex:cic-antimatter}
	The identity $1/25 + 3/50 = 1/10$ gives $\zz[1/10]_{\ge 0} = M_{\{1/10\}} \subseteq M_{\{1/25, 3/50\}}$. The reverse inclusion holds because $1/25 = 4/100$ and $3/50 = 6/100$ belong to the semiring $\zz[1/10]_{\ge 0}$. Therefore
	\[
		M_{\{1/25, 3/50\}} = \zz[1/10]_{\ge 0} = M_{\{1/10\}},
	\]
	which is a completely integrally closed antimatter monoid by Theorem~\ref{thm:cic}.
	\exampleqed
\end{example}

\bigskip
\section{Summary and Questions}
\label{sec:questions}

Table~\ref{tab:examples} locates our examples in Figure~\ref{fig:diagram}. For every finite subset $S$ of $\rr_{>0}$, Theorem~\ref{thm:fixed-length} reduces the chain of factorization implications to that shown in Figure~\ref{fig:rev-diagram}.

\begin{table}[htbp]
	\centering
	\small
	\begin{tabular}{@{}p{0.42\textwidth}p{0.48\textwidth}@{}}
		\toprule
		Generating set or monoid & Proved behavior \\
		\midrule
		Algebraically independent $S$ & UF, hence all six properties \\
		$\{t, t^2 - t + 1\}$ with $t > 1$ transcendental & HF and FF, but not UF \\
		$\{1 + \sqrt{2}, 1 + 2\sqrt{2}\}$ & HF and FF, but not UF \\
		$\{3/2\}$ & FF, BF, ACCP, and atomic, but not HF \\
		$L_{(3 + \sqrt{5})/2}$ & FF, BF, ACCP, and atomic, but not HF \\
		$\{2/3\}$ & Atomic, but not ACCP \\
		$\{2/3, 2/5\}$ and $\{3/2, 3/5\}$ & Nonempty set of atoms, but not atomic \\
		The pair in Example~\ref{ex:axes} & All pure powers are atoms, but not atomic \\
		$\{2/3, 3/4\}$ & Antimatter \\
		\bottomrule
	\end{tabular}
	\vspace{\abovecaptionskip}
	\caption{Positions of the examples of this paper relative to the diagram in Figure~\ref{fig:diagram}.}
	\label{tab:examples}
\end{table}

\begin{figure}[ht]
	\begin{tikzcd}
		\textbf{ UF } \ \arrow[r, Rightarrow] \arrow[red, r, Leftarrow, "/"{anchor=center,sloped}, shift left=1.7ex] & \ \textbf{ HF } \arrow[r, Rightarrow] \arrow[red, r, Leftarrow, "/"{anchor=center,sloped}, shift left=1.7ex] & \ \textbf{ FF } \ \arrow[r, Leftrightarrow]	& \ \textbf{ BF } \arrow[r, Rightarrow] & \textbf{ ACCP } \arrow[r, Rightarrow] \arrow[red, r, Leftarrow, "/"{anchor=center,sloped}, shift left=1.7ex]  & \textbf{ Atomic }
	\end{tikzcd}
	\caption{The implications in the diagram show inclusions of the six properties in Figure~\ref{fig:diagram} for the class of additive monoids of finitely generated positive semirings. Red arrows mark implications that are not reversible.}
    \label{fig:rev-diagram}
\end{figure}
\smallskip

The first, second, and last one-way implications are strict by Examples~\ref{ex:quadratic-hf}, \ref{ex:ff-not-hf}, and~\ref{ex:atomic-notaccp}. The remaining question is whether the ACCP implies BF. This holds for one-generator evaluations~\cite[Theorem~4.11]{CG22}, Laurent evaluations~\cite[Theorem~4.3]{sZ22}, and finite sets of positive rational generators by Theorem~\ref{thm:rational-middle}.

\begin{question} \label{q:accp-bf-ff}
	If $S$ is a finite subset of $\rr_{>0}$ and $M_S$ satisfies the ACCP, must $M_S$ be a BFM, and hence an FFM?
\end{question}

The rational examples also raise the problem of characterizing atomicity arithmetically.

\begin{question} \label{q:atomicity-criterion}
	Find a necessary and sufficient arithmetic criterion for atomicity of $M_S$ for finite $S \subseteq \qq_{>0}$.
\end{question}

When $S$ consists of algebraic elements, a concise criterion seems unlikely, and certainly none as simple as the single-generator case described in~\cite[Corollary~4.6]{CGLY26}.
\smallskip

Finally, for algebraic generators specified by polynomials and real roots, the condition in Corollary~\ref{cor:embedding} is decidable: enumerate the real embeddings of $\qq(S)$ and test $\sigma(s) \ge 1$ for every $s \in S$. The question of when more general additive weights exist remains open.

\begin{question} \label{q:certificates}
	Which finite sets $S$ of positive algebraic numbers admit an additive weight $w \colon M_S \to \rr_{\ge 0}$ and a constant $\delta > 0$ such that $w(x) \ge \delta$ for every nonzero $x \in M_S$? Can one decide existence from defining polynomials and specified roots and construct a weight when one exists?
\end{question}

\bigskip
\section*{Acknowledgments}

The authors would like to thank the PRIMES-USA program at MIT for supporting their collaboration.

\bigskip

\end{document}